\documentclass[11pt,oneside,reqno]{amsart}
\usepackage{amsmath,amsthm}
\usepackage{amssymb}
\usepackage{color}
\usepackage{enumerate}
\usepackage{algorithm}
\usepackage{algpseudocode}
\usepackage[margin=3.5cm]{geometry}
\usepackage{graphicx}
\usepackage{bm}
\usepackage{soul,xcolor}
\usepackage{cancel}

\newtheorem{thm}{Theorem}[section]
\newtheorem{theorem}{Theorem}[section]

\newtheorem{lem}[thm]{Lemma}
\newtheorem{prop}[thm]{Proposition}

\newtheorem{remark}[thm]{Remark}

\newtheorem{definition}[thm]{Definition}

\newtheorem{corollary}[thm]{Corollary}

\numberwithin{equation}{section}

\usepackage{hyperref}\usepackage{hyperref}

\def\C{ {\mathbb C} }
\def\R{ {\mathbb R} }

\def\Z{ {\mathbb Z} }
\def\N{ {\mathbb N} }

\def\T{ {\mathbb T} }

\def\Vbeta{V^\infty_{\beta,\lambda}}

\begin{document}
\title[]{Conjugate phase retrieval in shift-invariant spaces generated by a Gaussian}
\author{Yang Chen and Cheng Cheng  }

\address{Chen: Key Laboratory of Computing and Stochastic Mathematics (Ministry of Education), School of Mathematics and Statistics,
 Hunan Normal University, Changsha, Hunan 410081, P. R. China,
 email: ychenmath@hunnu.edu.cn}

\address{Cheng:
School of Mathematics, Sun Yat-sen University, Guangzhou, Guangdong,
510275, China, email: chengch66@mail.sysu.edu.cn}

\thanks{This project is partially supported by  National Key RD Program of China (No. 2024YFA1013703),  National Natural
Science Foundation of China (12171490),  Fundamental Research Funds for the Central Universities, Sun Yat-sen University (24lgqb019),  Research Foundation of Education Bureau of Hunan Province, China (24B0107) and  Natural
Science Foundation of Hunan Province, China (2025JJ50008).}
\maketitle
\begin{abstract}
Conjugate phase retrieval aims to recover a function, up to a unimodular constant and  conjugation, from its magnitude measurements. In this paper, we investigate conjugate phase retrieval in a shift-invariant space generated by a Gaussian function, which is a phaseless sampling and reconstruction problem from phaseless Hermite samples, namely the magnitude measurements of a function and its derivative taken on a discrete sampling set. We first show that the modulus of any function in the Gaussian shift-invariant space, as well as the modulus of its derivative, can be determined from its phaseless Hermite samples on a discrete set. We then prove that any function in this Gaussian-generated shift-invariant space can be uniquely recovered, up to a unimodular constant and conjugation, from its phaseless Hermite samples
 on a discrete set. For functions whose coefficient sequences are finitely supported, we further provide an explicit reconstruction procedure.  
\end{abstract}

\section{Introduction}

\medskip 
Shift-invariant spaces play a fundamental role in wavelet analysis,  approximation theory, and signal processing. Given a generator $\phi$, the associated shift-invariant space is defined by
\begin{equation}\label{sis.def.intro}
V(\phi)
=
\Big\{
\sum_{k\in\mathbb Z} c_k \, \phi(\cdot-k)
:\ c_k\in\mathbb C
\Big\}.
\end{equation}
Such spaces provide a rigorous functional framework for modeling and
approximating signals with prescribed spectral characteristics and structural
constraints
\cite{AG01, AST05,Christensenbook, daubechiesbook, DDR94, jia92, mallatbook,unser00}.  

 Phase retrieval has been studied in  various fields, including X-ray crystallography, diffraction imaging, optics, and related areas.  It considers the recovery of a signal from the magnitudes of its linear measurements, up to a trivial ambiguity \cite{BCE06, bunk2007diffractive, F78,harrison1993phase,HES16,H01,JEH16,millane90,walther1963question}. 
	Phase retrieval in a shift-invariant space is an infinite-dimensional phaseless sampling and reconstruction problem, where one seeks to recover a function from magnitude samples on a continuous domain or a discrete sampling set \cite{Chen2020,Cheng2021,Cheng2019, CS2021, grochenig2020,Romero2018,shenoy2016exact,Thakur11}. 

    Phase retrieval for real-valued signals has been established in several classes of shift-invariant spaces.
Thakur showed that a real-valued bandlimited function can be uniquely determined, up
to a global sign, from its magnitude samples with sampling rate more than twice Nyquist rate \cite{Thakur11}. 
The bandlimited functions belong to the shift-invariant space generated by the sinc
function $\frac{\sin(\pi \cdot)}{\pi \cdot}$. Later, 
Gr\"ochenig proved that every real-valued function in a shift-invariant space
generated by a Gaussian $\phi_\lambda := e^{-\lambda \cdot^2}$ can be recovered, up to
a sign, from their phaseless samples on a separated set with  lower Beurling density greater than 2, and also showed that phase retrieval
is impossible for complex-valued signals in the same space \cite{grochenig2020}. 
Related phaseless sampling and reconstruction problems for real-valued signals, often
referred to as sign retrieval, were further studied by Romero for shift-invariant
spaces generated by totally positive functions \cite{Romero2018}.

A different phenomenon arises when the generator $\phi$ is a compactly supported real-valued function. Not all real-valued functions in $V(\phi)$ are phase retrievable.  In \cite{Chen2020, Cheng2019, CS2021}, the notion of nonseparability was
introduced to characterize phase retrievability in $V(\phi)$.  Specifically, a real-valued function $f\in V(\phi)$ is said to be nonseparable if
there do not exist nonzero functions $f_1,f_2\in V(\phi)$ such that
\[
f = f_1 + f_2, \ \ 
f_1 f_2 = 0.
\]
Equivalently, $f$ cannot be decomposed into two nontrivial components supported on disjoint sets. It was shown in \cite{Chen2020, Cheng2019, CS2021} that a real-valued function in $V(\phi)$ is phase retrievable from its magnitude samples if and only if it is nonseparable, and that such functions admit reconstruction, up to a global sign,
from phaseless samples taken on a discrete set with finite sampling density.

Let $\phi$ be a real-valued generator and consider the complex shift-invariant space $V(\phi)$. Then $V(\phi)$ is invariant under complex conjugation, that is a function
$f \in V(\phi)$ belongs to the space if and only if its complex conjugate
$\overline f$ also belongs to $V(\phi)$. 
Since $f$ and $\overline f$ satisfy
$|f(x)| = |\overline f(x)|$ for all $x \in \mathbb R$, magnitude  measurements
cannot distinguish between them, which makes classical phase retrieval impossible in
this setting. To address this intrinsic ambiguity, the notion of
\emph{conjugate phase retrieval} was introduced
\cite{Chen2022, CW23, Cheng2025, Lai, LLW21, Mcdonald}.
A linear space $\mathcal C$ is said to be conjugate invariant if
$\overline f \in \mathcal C$ whenever $f \in \mathcal C$. 
Given a conjugate-invariant space $\mathcal C$ and a collection of linear functionals $\Psi$, conjugate phase retrieval aims to determine a signal
$f \in \mathcal C$, up to a unimodular constant and conjugation, from its phaseless
measurements $\{ |\psi(f)| : \psi \in \Psi \}$.
More precisely, $f$ is conjugate phase retrievable if
\[
g \in \Xi_{f,\Psi}
\ \text{if and only if} \
g = \alpha f \ \text{or} \ g = \alpha \overline f
\ \text{for some } \alpha \in \mathbb T,
\]
where
\[
\Xi_{f,\Psi}
:= \{ g \in \mathcal C : |\psi(g)| = |\psi(f)|,\ \psi \in \Psi \},
 \ {\rm and } \ 
\mathbb T := \{ \alpha \in \mathbb C : |\alpha| = 1 \}.
\]

In \cite{Chen2022}, it was shown that not all signals in a conjugate-invariant space
are conjugate phase retrieval from magnitude only measurements.
Similar to the characterization of phase retrievability by non-separability in real
linear spaces, the authors of \cite{Chen2022} established a general characterization
of conjugate phase retrieval from magnitude measurements $|\Psi(f)|:=\{|f(x)|, x\in\Gamma\}$, where $\Gamma$ is either the whole domain  of $f$ or its subset.  See also \cite{Lai} for finite-dimensional results arising from generic real frame vectors, 
and \cite{CW23} for conjugate phase retrieval in complex shift-invariant spaces
generated by compactly supported real-valued functions.

Structured phaseless measurements have been introduced to enable (conjugate) phase retrieval for complex-valued signals 
\cite{Chen2022, CQ22, Cheng2025, jaming14, LLW21,Mcdonald,  pohl2014}.
For instance, McDonald \cite{Mcdonald} proved that any complex bandlimited function
that extends to an entire function of finite order can be uniquely determined, up to
a unimodular constant and conjugation, from the joint magnitude measurements
$|f(x)|$ and $|f'(x)|$ on the whole real line.
Conjugate phase retrieval for complex-valued bandlimited signals from phaseless
structured convolutions was studied in \cite{LLW21}. In \cite{CQ22}, conjugate phase retrieval from linear canonical transforms was established for functions in
shift-invariant spaces generated by B-splines. The work \cite{Chen2022} considered the determination of vector-valued functions,
up to an orthogonal matrix, from structured phaseless measurements on graphs,  and conjugate phase retrieval for
complex-valued graph signals using magnitudes on vertices together with relative
magnitudes on edges was investigated in \cite{Cheng2025}.

Motivated by these results, the present paper investigates conjugate phase retrieval in the complex
shift-invariant space generated by the Gaussian function $\phi_\lambda= e^{-\lambda \cdot^2}$,
\begin{equation}\label{sis.def}
V^\infty_{\beta,\lambda}
:= \Big\{
f = \sum_{k \in \mathbb{Z}} c_k e^{-\lambda(\cdot - \beta k)^2}
:\ \{c_k\}_{k \in \mathbb{Z}} \in \ell^\infty(\mathbb{Z})
\Big\},
\end{equation}
where the step size $\beta > 0$ and the parameter $\lambda > 0$. Hermite sampling, which incorporates both function values and derivative values at
sampling points, has been widely used in applications such as aircraft instrument
communication and air traffic control simulation \cite{CQ22, CW23, GRS2020, LLW21,  Mcdonald}.
Inspired by this idea, we consider conjugate phaseless sampling and reconstruction
in the shift-invariant space $V^\infty_{\beta,\lambda}$ from the magnitude samples
\[
|f(\gamma)| \ \text{and} \ |f'(\gamma)|, \ \gamma \in \Gamma,
\]
taken on a discrete set $\Gamma \subset \mathbb{R}$. Our main result, stated in Theorem \ref{main.thm}, shows that any function
$f \in V^\infty_{\beta,\lambda}$ satisfying a mild assumption can be uniquely
determined, up to a unimodular constant and conjugation, from the phaseless Hermite
samples $|f(\gamma)|$ and $|f'(\gamma)|$, $\gamma \in \Gamma$, provided that the
sampling set $\Gamma$ has lower Beurling density
$D^{-}(\Gamma) > 2\beta^{-1}$.
Hermite functions extend the Gaussian function by preserving its  analytic properties while forming a complete orthonormal basis, and they play a central role in time–frequency analysis \cite{abreu2014, cohenbook, szego75}.  Motivated by these features, we consider shift-invariant spaces generated by Hermite functions. In Theorem~\ref{hermite.cpr.prop}, we establish that any function in a shift-invariant space generated by a Hermite function admits conjugate phase retrieval from its Hermite phaseless samples taken on any sequentially compact infinite sampling set.

  This paper is organized as follows. In Section \ref{phaseless.sec}, we show that two functions in the shift-invariant space $\Vbeta$  have the same phaseless Hermite samples on some discrete set if and only if they have the same  pointwise magnitudes everywhere on the whole real line, see Theorem \ref{cpr.sis.thm1}.  In Section \ref{cpr.sec}, we show that  the functions in $\Vbeta$,  whose coefficient sequences satisfy suitable decay
conditions, are conjugate phase retrieval from its phaseless Hermite samples on a sampling set with lower Beurling density 
$D^{-}(\Gamma)> 2\beta^{-1}$, see Theorem \ref{main.thm}. We also show that functions in the shift-invariant space generated by a Hermite function admit conjugate phase retrieval from their Hermite phaseless samples on a  sequentially compact infinite set. 
In Section \ref{spec.sec}, motivated by practical considerations,  we consider the conjugate phase retrieval of the signals in $\Vbeta$ with finite coefficient sequence, and  present an explicit reconstruction algorithm  for  signals in $V^\infty_{1,1}$. 
The proofs of 
technical lemmas are collected in Appendix \ref{lem.pr.sis}-\ref{appendix.sec}.

\smallskip

{\bf  Notation}:  In this paper, $\ast$ denotes the classical convolution operator.
A \emph{multiset} is a generalization of  a set in which elements are allowed to appear with multiplicity.  For a multiset $V$, the multiplicity of an element $x$ in $V$, denoted by $m_V(x)$, is a nonnegative integer  representing the number of times $x$ occurs in $V$.  
The multiset union, intersection, and difference of the multisets $V$ and $W$ are denoted by
$V \dot\cup W$, $V \dot\cap W$, and $V \dot\setminus W$, respectively.
These operations are defined through their multiplicity functions as follows:
for every element $x$ in $V$ or $W$, 
\begin{equation*}
m_{V\dot\cap W}(x)
= \min\{m_V(x),\, m_W(x)\},
\end{equation*}
\begin{equation*}
m_{V\dot\cup W}(x)
= \max\{m_V(x),\, m_W(x)\},
\end{equation*}
and
\begin{equation*}
m_{V\dot\setminus W}(x)
= \max\{m_V(x)-m_W(x),\,0\}.
\end{equation*}

\medskip

\section{Phaseless Hermite sampling in a shift-invariant space with Gaussian }\label{phaseless.sec}
\medskip

Given a signal  $f\in \Vbeta$ in \eqref{sis.def},  let
\begin{equation}\label{mf.def}
{\mathcal M}_{f}:=\{g\in \Vbeta: \   |g(x)|=|f(x)| {\rm\ and}\  |f'(x)|=|g'(x)| \ {\rm for \ all} \  x\in\R\}
\end{equation}
contain all functions $g\in  \Vbeta$ having the same Hermite magnitude measurements on the whole line  $\R$ as $f$ has. The phaseless Hermite sampling and reconstruction problem considers whether there exists a discrete sampling set $\Gamma$ such that
 \begin{equation*}\label{cpr.equal.def}
 {\mathcal M}_{f,\Gamma}=  {\mathcal M}_{f},
 \end{equation*}
where
\begin{equation}\label{mf.gamma.def}{\mathcal M}_{f,\Gamma}=\{g\in \Vbeta: \   |g(\gamma)|=|f(\gamma)| {\rm\ and}\  |f'(\gamma)|=|g'(\gamma)| \ {\rm for \ all} \  \gamma\in \Gamma\}.
\end{equation}

In this section, we show that for  any signal $f=\sum_{k\in \Z} c_ke^{-\lambda(\cdot-\beta k)^2}\in \Vbeta$   with the mild condition
\begin{equation}\label{signal.assump}
\{kc_k\}_{k\in \Z} \in \ell^\infty(\Z),
\end{equation}
 there exists a discrete sampling set $\Gamma$ with the lower Beurling density $D^-(\Gamma)>2\beta^{-1}$ such that the equality
 $${\mathcal M}_{f} =  {\mathcal M}_{f,\Gamma}$$ holds,  see  Theorem \ref{cpr.sis.thm1}.

A set $\Gamma\subset\R$  is said to have the lower Beurling density $D^{-}(\Gamma)$ if 
\begin{equation}\label{beurling.eq}
D^{-}(\Gamma):=\liminf_{r\rightarrow\infty} \inf_{x\in\R}\frac{\#\big(\Gamma\cap[x-r,x+r]\big)}{2r},
\end{equation}  see \cite{AG01,AST05,Cheng2019,CS2021,grochenig2020,GRS2018} and references therein. To illustrate this notion, consider the sampling set
\[
\Gamma=\{k,\;k+3^{-|k|},\;k+2^{-|k|}\}_{k\in\mathbb Z}.
\]
Except for finitely many overlaps, this set contains three sampling points in each
unit interval $[k,k+1)$ for large $|k|$. A direct counting argument shows that
\[
D^{-}(\Gamma)=3.
\]


\begin{theorem}\label{cpr.sis.thm1}
{\rm 
Let $f$  be a function in the shift-invariant space $V^\infty_{\beta, \lambda}$ satisfying \eqref{signal.assump}, and $\Gamma$ be a sampling set with the lower Beurling density $D^-(\Gamma)>2\beta^{-1}$. Then 
\begin{equation*}
M_{f,\Gamma}=M_f, 
\end{equation*} 
that is, the samples of $|f|$ and $|f'|$ on $\Gamma$ uniquely determine the
entire modulus functions $|f|$ and $|f'|$ on $\mathbb R$.} 
\end{theorem}

From the above theorem, we know that there exists a discrete set $\Gamma \subset \mathbb{R}$ of finite density such that
the modulus $|f|$ and $|f'|$ can be uniquely determined everywhere on $\mathbb{R}$  from the samples of $|f|$ and $|f'|$ on $\Gamma$.

\smallskip

  To prove Theorem \ref{cpr.sis.thm1}, we recall some basic results in the shift-invariant space generated by a Gaussian.

\begin{lem}\label{f.abs.lem}{\rm[Lemma 3, \cite{grochenig2020}]}
{\rm If $f=\sum_{k\in\Z}c_k e^{-\lambda(\cdot-\beta k)^2}\in V^\infty_{\beta, \lambda}$, 
 then $|f|^2\in V^\infty_{\frac{\beta}{2}, 2\lambda}$.
}
\end{lem}
The above theorem on the function $|f|^2, f\in \Vbeta$ is crucial in the  (conjugate) phase retrieval  of a function  in the complex shift-invariant space generated by a Gaussian. 
For the completeness, we include a proof in the Appendix \ref{lem.pr.sis}. 

 \begin{lem}\label{sampling.lem}{\rm[Theorem 4.4 (a), \cite{GRS2018}]}
{\rm 
 If the set $\Gamma\subset\R$ has the lower Beurling density $D^-(\Gamma)>\beta^{-1}$, then $\Gamma$ is a uniqueness set for $V_{\beta,\lambda}^{\infty}$.

  }
\end{lem}


From the above lemmas, we remark that the modulus function $|f|$ can be uniquely determined from the phaseless samples of $|f|$ taken on  a sampling set $\Gamma$ with lower Beurling density $D^-(\Gamma)> 2\beta^{-1}$.

Now we are ready for the proof of Theorem \ref{cpr.sis.thm1}.
\begin{proof}[Proof of Theorem \ref{cpr.sis.thm1}]
Without loss of generality, we set $\beta=1$, as  $f \in V^\infty_{\beta,\lambda}$
if and only if  $f\in V^\infty_{1,\beta^{2}\lambda}$. 

Write $V^\infty_{1,\lambda} := V^\infty_{\beta=1,\lambda}$.  Denote $\tilde{c}_k=c_ke^{-\lambda k^2}$ and $r_k=\tilde c\ast \bar {\tilde c}(k)e^{\lambda k^2/2}$ for all $k\in\Z$. 
By Lemmas \ref{f.abs.lem} and \ref{sampling.lem}, we know that the coefficients $r_k, k\in\Z$, can be uniquely determined from $|f(\Gamma)|$ as the modulus 
\begin{equation}\label{f.mod}
|f(x)|^2=\sum_{k\in\Z} r_ke^{-2\lambda(x-k/2)^2}\in V^\infty_{\frac{1}{2}, 2\lambda}.
\end{equation}
 Next we  prove that $|f'|$ can be uniquely determined from $|f(\Gamma)|$ and $|f'(\Gamma)|$.

 Write $\omega(x):=\sum_{k\in\Z}kc_ke^{-\lambda(x-k)^2}$.  By \eqref{signal.assump} and  Lemma \ref{f.abs.lem}, we have that
 \begin{equation}\label{h2.def}
\omega \in V^{\infty}_{1,\lambda} \ \ {\rm and} \ \
|\omega|^2\in V^{\infty}_{\frac{1}{2}, 2\lambda}.
\end{equation}
Observe that
\begin{equation*}\label{f'.def}
f'(x)=2\lambda\sum_{k\in\Z}(k-x)c_ke^{-\lambda(x-k)^2}=2\lambda(\omega(x)-xf(x)).
\end{equation*}
By a direct calculation, we have 
\begin{eqnarray}\label{f'.def.1}
\frac{|f'(x)|^2}{4\lambda^2}\hskip-0.05in&\hskip-0.1in=&\hskip-0.1in x^2|f(x)|^2+|\omega(x)|^2 -2x\Re( \omega(x)\overline{f(x)})\nonumber\\
\hskip-0.05in&\hskip-0.1in=&\hskip-0.1in x^2|f(x)|^2+|\omega(x)|^2 -2x\Re( \sum_{k,j\in\Z}kc_k\bar{c}_je^{-\lambda((x-k)^2+(x-j)^2)})\nonumber\\
\hskip-0.05in&\hskip-0.1in=&\hskip-0.1in x^2|f(x)|^2+|\omega(x)|^2 -x\sum_{k,j\in\Z}(k+j)c_k\bar{c}_je^{-\lambda((x-k)^2+(x-j)^2)}\nonumber\\
\hskip-0.05in&\hskip-0.1in=&\hskip-0.1in x^2|f(x)|^2+|\omega(x)|^2 -x\sum_{n\in\Z}nr_ne^{-2\lambda(x-\frac{n}{2})^2}.
\end{eqnarray}
Then the evaluations of $|\omega|$ on the discrete set $\Gamma$ 
\begin{equation*}\label{h.sample}
|\omega(\gamma)|^2=\frac{|f'(\gamma)|^2}{4\lambda^2}+\gamma\sum_{n\in\Z}nr_ne^{-2\lambda(\gamma-\frac{n}{2})^2}-\gamma^2|f(\gamma)|^2, \gamma\in \Gamma,
\end{equation*}
can be obtained by the coefficients $r_k, k\in \Z$ and the phaseless Hermite samples $|f(\gamma)|$ and $|f'(\gamma)|, \gamma\in\Gamma$.
Together with Lemma \ref{sampling.lem} and \eqref{h2.def}, we know that the  function $|\omega|^2$ can be uniquely determined. 
Hence $|f'|$ can be uniquely determined by \eqref{f'.def.1}.

Then for any signal $g\in {\mathcal M}_{f,\Gamma}$, i.e., $|g(\gamma)|=|f(\gamma)|$ and  $|g'(\gamma)|=|f'(\gamma)|$, $\gamma\in \Gamma$,  we have $|g(x)|=|f(x)|$ and $|g'(x)|=|f'(x)|$ for all $x\in \R$ and hence $g\in {\mathcal M}_f$. This together with ${\mathcal M}_{f}\subset {\mathcal M}_{f, \Gamma}$ proves our conclusion. 
\end{proof}

%
%

\medskip
\section{Conjugate phaseless sampling and reconstruction}\label{cpr.sec}

\medskip

 For any signal  $f\in \Vbeta$, and   $\alpha \in\T$,  the signals $\alpha f, \alpha \overline{f}$ are in the shift-invariant spaces $\Vbeta$  and have the same Hermite magnitude measurements  as $f$ has on the whole line $\R$. Then we have 
$$\{\alpha f, \alpha \overline{f}:\alpha\in \mathbb T\} \subset {\mathcal M}_f. $$
The conjugate phase retrieval considers whether the converse of the above inclusion is true, that is, whether $\mathcal{M}_f$ consists only of unimodular multiples of $f$ and its conjugate. 

In this section, we show that this is indeed the case under mild assumptions. More precisely, we prove that any function
\[
f(x)=\sum_{k\in\Z} c_k e^{-\lambda(x-\beta k)^2}\in \Vbeta
\]
with $\{k c_k\}_{k\in\Z}\in\ell^\infty(\Z)$ can be uniquely determined, up to a unimodular constant and conjugation, from the phaseless samples
\[
\{|f(\gamma)|, |f'(\gamma)|:\gamma\in\Gamma\},
\]
taken on a discrete set $\Gamma\subset\R$ with lower Beurling density $D^-(\Gamma)> 2\beta^{-1}$. Equivalently,
\begin{equation}\label{cpr.def}
\mathcal{M}_f=\{\alpha f, \alpha\overline{f}: \alpha \in\T\}.
\end{equation}

\begin{thm}\label{main.thm}
{\rm Let  $f(x)=\sum_{k\in\Z}c_ke^{-\lambda(x-k\beta)^2}$ be a function in $\Vbeta$ satisfying \eqref{signal.assump}, and   $\Gamma\subset \R$ be a  discrete set with the lower Beurling density $D^-(\Gamma)>2\beta^{-1}$. 
 Then $f$
  can be uniquely determined, up to a unimodular constant and conjugation, from its phaseless Hermite samples $|f(\gamma)|$ and $|f'(\gamma)|, \gamma\in \Gamma$.}
\end{thm}

Combined with Theorem~\ref{cpr.sis.thm1}, the proof of Theorem~\ref{main.thm} reduces to showing that the coefficient sequence $c=\{c_k\}_{k\in\Z}$ can be determined, up to a unimodular constant and conjugation, from $|f|^2$ and $|f'|^2$.

\smallskip 
In \cite{grochenig2020}, it is shown that the Fourier series of a sequence with Gaussian decay can be extended to an entire function of order $2$. For convenience, we recall the result below.

\begin{lem}\label{extension.lem}{\rm [Lemma 2, \cite{grochenig2020}]}
{\rm For a sequence $\{{\tilde c}_k=c_ke^{-\lambda k^2}\}_{k\in \Z}$ with some $c\in \ell^\infty(\Z)$ and $\gamma>0$,  the Fourier series $\hat {\tilde c}(\xi):=\sum_{k\in \Z}{\tilde c}_k e^{2\pi ik\xi}$ can be extended to an entire function $\tilde C(z) =\hat {\tilde c}(\xi+iy)$ of order 2, that is, $|\tilde C(\xi+iy)|=O(e^{\pi^2y^2/\lambda})$.}
\end{lem}

Recall that an entire function $f$ is said to be of finite order $\rho$ if 
\begin{equation*}
\rho:=\lim_{r \to +\infty}\sup\frac{\log\log M_{f}(r)}{\log r}<\infty,
\end{equation*}
where $r>0$ and $M_{f}(r)=\max_{|z|=r}|f(z)|$.

Let $f$ be an entire function of finite order, and denote by $Z_f$ the multiset
of its nonzero zeros, counted with multiplicities. Define the involution of $f$
by
\[
f^\natural(z) := \overline{f(-\bar z)},
\]
and the associated involution multiset by
\[
Z_f^\natural := -\overline{Z_f}.
\]
Then $(f^\natural)^\natural = f$. Moreover, the function $f^\natural$ is also
entire and has the same order as $f$, and its zero multiset satisfies
\[
Z_{f^\natural} = Z_f^\natural.
\]

The following lemma plays a crucial role  in our proof of Theorem \ref{main.thm}.


\begin{lem}\label{pr.entire.lem}
{\rm Let $f,g $ be entire functions of  finite order  and satisfy
 \begin{equation}\label{f.lemma.eq}
 ff^\natural=gg^\natural{\rm \  and\ }f'(f')^\natural=g'(g')^\natural.\end{equation} Then we have
$f=\alpha g$ or $f=\alpha g^\natural$, where $\alpha \in \T$ is some unimodular constant.}
\end{lem}
 The above result has been established in \cite[Theorem 1]{Mcdonald} for entire functions where the involution function of the function $f$ is defined by $f_\ast(z)=\overline{f(\bar z)}$.  Our proof is similar to the proof there. For completeness, a proof of Lemma~\ref{pr.entire.lem} is provided in Appendix~\ref{appendix.sec}.

 \vskip0.06in

 Now we are ready to prove the main theorem. 

  \vskip0.06in

 \begin{proof}[ Proof of Theorem \ref{main.thm}:]  
To simplify notation, we set $\beta = 1$. Let
\[
f(x)=\sum_{k\in \mathbb{Z}} c_{1,k}\, e^{-\lambda(x-k)^2} \in V^\infty_{1, \lambda}
\ \ \text{and} \ \ 
g(x)=\sum_{k\in \mathbb{Z}} c_{2,k}\, e^{-\lambda(x-k)^2} \in V^\infty_{1, \lambda}, 
\]
with coefficient sequences
\[
c_1=\{c_{1,k}\}_{k\in \mathbb{Z}}
\ \  \text{and} \ \ 
c_2=\{c_{2,k}\}_{k\in \mathbb{Z}}
\]
satisfying \eqref{signal.assump}. Assume that
\[
|f(\gamma)| = |g(\gamma)| \ \ \text{and} \ \ 
|f'(\gamma)| = |g'(\gamma)|, \ \gamma \in \Gamma,
\]
where $\Gamma$ is a sampling set with lower Beurling density $D^{-}(\Gamma) > 2$.

Define
\[
\omega_1(x)=\sum_{k\in \mathbb{Z}} d_{1,k}\, e^{-\lambda(x-k)^2}
\ \text{and} \
\omega_2(x)=\sum_{k\in \mathbb{Z}} d_{2,k}\, e^{-\lambda(x-k)^2},
\]
where the coefficients are given by
\[
d_{l,k} = kc_{l,k}, \ k\in \mathbb{Z}, \; l\in \{1,2\}.
\]

Denote
\[
\tilde{c}_{l,k} := c_{l,k} e^{-\lambda k^{2}}
\quad \text{and} \quad
\tilde{d}_{l,k} := d_{l,k} e^{-\lambda k^{2}},
\qquad l\in\{1,2\}, \; k\in\mathbb{Z}.
\]
Set
\begin{equation}\label{coeff.eq.1}
r_{l,k}
:= (\tilde c_l \ast \overline{\tilde c_l})(k)\, e^{\lambda k^{2}/2},
\end{equation}
and
\begin{equation}\label{coeff.eq.2}
s_{l,k}
:= (\tilde d_l \ast \overline{\tilde d_l})(k)\, e^{\lambda k^{2}/2},
\end{equation}
for $l\in\{1,2\}$ and $k\in\mathbb{Z}$.

Following  the same argument as in the proof of Theorem~\ref{cpr.sis.thm1}, the phaseless
samples $\{|f(\gamma)|,\ |f'(\gamma)|: \gamma \in \Gamma\}$ uniquely determine
$\{r_{1,k}\}_{k\in\mathbb{Z}}$ and $\{s_{1,k}\}_{k\in\mathbb{Z}}$, and the samples
$\{|g(\gamma)|, \ |g'(\gamma)|: \gamma \in \Gamma\}$ uniquely determine
$\{r_{2,k}\}_{k\in\mathbb{Z}}$ and $\{s_{2,k}\}_{k\in\mathbb{Z}}$.

Since $|f(\gamma)| = |g(\gamma)|$ and $|f'(\gamma)| = |g'(\gamma)|$ for all
$\gamma \in \Gamma$, it follows from Theorem~\ref{cpr.sis.thm1} that
\[
r_{1,k} = r_{2,k}
\quad \text{and} \quad
s_{1,k} = s_{2,k},
\qquad \text{for all } k \in \mathbb{Z}.
\]

Define
\[
R_l(z) := \sum_{k\in\mathbb{Z}} r_{l,k}\, e^{-\lambda k^{2}/2} e^{2\pi i k z}
\quad \text{and} \quad
S_l(z) := \sum_{k\in\mathbb{Z}} s_{l,k}\, e^{-\lambda k^{2}/2} e^{2\pi i k z},
\qquad l\in\{1,2\}.
\]
Then it follows immediately that
\begin{equation}\label{coeff.eq.3}
R_1(z) = R_2(z)
\quad \text{and} \quad
S_1(z) = S_2(z),
\qquad z \in \mathbb{C}.
\end{equation}

   Next we show that   the coefficient sequences $c_1$ and  $c_2$ can be determined, up to a unimodular constant and conjugation, from the sequences $\{r_k\}_{k\in\Z}$ and $\{s_k\}_{k \in\Z}$.

From a direct calculation, we have
\begin{equation*}
R_l(z)=C_l(z)C_l^\natural(z) \ {\rm  \ and \ } \ S_l(z)=D_l(z)D_l^\natural(z)=-\frac{1}{4\pi^2}C_l'(z)\big(C_l'(z)\big)^\natural,  \ \ l\in \{1, 2\}, 
\end{equation*}
where \begin{equation*}
C_l(z)=\sum_{k\in\Z} \tilde c_{l, k}e^{2\pi ikz} \ {\rm \ and \ } \ D_l(z)=\sum_{k\in\Z} \tilde d_{l, k}e^{2\pi ikz}.
\end{equation*}
From \eqref{coeff.eq.3}, we have
\begin{equation}\label{coeff.equal.eq}
C_1(z)C_1^\natural(z)=C_2(z)C_2^\natural(z) \ {\rm  \ and \ } \ C_1'(z)\big(C_1'(z)\big)^\natural=C_2'(z)\big(C_2'(z)\big)^\natural. 
\end{equation}
   By \eqref{signal.assump} and Lemma \ref{extension.lem}, the functions $C_l(z)$, $D_l(z)$ and  their involution functions $C_l^\natural(z), D_l^\natural(z), l=1,2$ 
  are entire functions of order 2. 
Together with \eqref{coeff.equal.eq},  Lemma \ref{pr.entire.lem} implies that  $C_1=\alpha C_2$ or $C_1=\alpha C_2^\natural$ for some $\alpha \in \T$.  By the linear independence of $\{e^{2\pi i  kz}\}_{k\in \Z}$,  the sequence $\{c_k\}_{k\in \Z}$ or $\{\bar{c}_k\}_{k\in \Z}$ can be determined, up to a unimodular constant, which completes our proof.
\end{proof}

\subsection{Conjugate phase retrieval in shift-invariant spaces generated by a Hermite function}

In this subsection, we study the conjugate phase retrieval in  shift-invariant spaces  generated by a Hermite function. We begin by recalling the definition of the Hermite functions and some basic
properties that will be used throughout this subsection \cite{abreu2014, cohenbook, szego75}.

\begin{definition}{\rm (Hermite functions)}
{\rm  For $n\in\mathbb N$, the $n$-th Hermite function is defined by
\begin{equation}\label{hermite.def}
h_n(x)=\zeta_n H_n(x)e^{-x^2/2},
\end{equation}
where $\zeta_n$ is a normalization constant and
\[
H_n(x)=(-1)^n e^{x^2}\frac{{\rm d}^n}{{\rm d}x^n}\big(e^{-x^2}\big)
\]
is the Hermite polynomial of order $n$. In particular, $h_0(x)$ coincides with the
Gaussian function up to normalization.}
\end{definition}

The Hermite polynomials satisfy the identity
\[
H_n'(x)=2nH_{n-1}(x),
\]
and $\deg H_n=n$. Moreover, for each $n\ge 0$, the family
$\{H_0,\ldots,H_n\}$ forms a basis of the space $\mathcal P_n$ of real
polynomials of degree at most $n$.

\medskip

For a fixed $n\ge 0$, we consider the shift-invariant space generated by $h_n$,
defined as
\begin{equation}\label{h.sis}
V(h_n)
:=
\left\{
\sum_{k\in\mathbb Z} c_k h_n(\cdot-k): \;
\{c_k\}_{k\in\mathbb Z}\in\ell^\infty(\mathbb Z)
\right\}.
\end{equation}

A key observation is that any function in $V(h_n)$ admits a representation in
terms of derivatives of Gaussian shifts. Indeed, since
$\{H_l(x/\sqrt2)\}_{l=0}^n$ forms a basis of $\mathcal P_n$, there exist constants
$m_l$, $0\le l\le n$, such that
\begin{equation}\label{hermite.expansion}
H_n(x)=\sum_{l=0}^n m_l H_l(x/\sqrt2).
\end{equation}
Together with the identity
\begin{equation}\label{gaussian.derivative}
\frac{{\rm d}^l}{{\rm d}x^l}\big(e^{-(x-k)^2/2}\big)
=(-1)^l(\sqrt2)^{-l}
H_l\!\left(\frac{x-k}{\sqrt2}\right)e^{-(x-k)^2/2},
\end{equation}
this implies that for any $f\in V(h_n)$,
\begin{equation}\label{hermite.sis.expansion}
f(x)
=
\zeta_n\sum_{l=0}^{n}(-1)^l(\sqrt2)^l m_l
\sum_{k\in\mathbb Z} c_k
\frac{{\rm d}^l}{{\rm d}x^l}\big(e^{-(x-k)^2/2}\big).
\end{equation}

The above expansion shows that $f$ is a finite linear combination of derivatives
of a Gaussian shift-invariant signal.

\medskip

 In the following proposition, we show that any function in $V(h_n)$ can be extended to an entire function on the complex plane.

\begin{prop}\label{hermite.lem}
{\rm Let $f\in V(h_n)$ be a function in the shift-invariant space generated by the Hermite function $h_n$ of order $n$ in \eqref{h.sis}.  Then the function $f$  and its squared modulus function $|f|^2$ can be extended to  the entire function of order at most 2.  }
\end{prop}

\begin{proof} 
Let $f\in V(h_n)$. By the expansion \eqref{hermite.sis.expansion}, the function $f$ can be written
as a finite linear combination of derivatives of a Gaussian shift-invariant signal. 
Set
\[
F(z):=\sum_{k\in\mathbb Z} c_k e^{-(z-k)^2/2}, \qquad z=x+iy\in\mathbb C,
\]
where $\{c_k\}_{k\in\mathbb Z}\in\ell^\infty(\mathbb Z)$.

For each $k\in\mathbb Z$,
\[
|e^{-(z-k)^2/2}|
= e^{(y^2-(x-k)^2)/2}
\le e^{y^2/2} e^{-(x-k)^2/2}.
\]
Hence, for any strip $\{z\in\mathbb C:|y|\le R\}$,
\[
|F(z)|
\le \|c\|_\infty e^{R^2/2}
\sum_{k\in\mathbb Z} e^{-(x-k)^2/2}.
\]
Since the sum $\sum_{k\in\mathbb Z} e^{-(x-k)^2/2}$ is uniformly bounded in
$x\in\mathbb R$, the series defining $F$ converges absolutely and uniformly on compact subsets of $\mathbb C$. Consequently, $F$ is entire and satisfies
\[
|F(z)|\le C e^{|z|^2},
\]
so that $F$ is an entire function of order at most $2$.

Moreover, it can be infinitely differentiable. For any $l\ge 0$,
\[
F^{(l)}(z)=\sum_{k\in\mathbb Z} c_k
\frac{{\rm d}^l}{{\rm d}z^l}\big(e^{-(z-k)^2/2}\big),
\]
and  each $F^{(l)}$ is also an entire function of order at most $2$.
Combining with \eqref{hermite.sis.expansion}, we conclude that $f$ admits the analytic extension
\begin{equation}\label{f.rep}
f(z)=\zeta_n\sum_{l=0}^n (-1)^l(\sqrt2)^l m_l\,F^{(l)}(z),
\end{equation}
which is an entire function of order no more than $2$. 

\smallskip 

Define $G(z):=f(z)\,\overline{f(\overline{z})}$. Since both
$f(z)$ and $z\mapsto\overline{f(\overline{z})}$ are entire, $G$ is entire. On the real axis,
\[
G(x)=f(x)\overline{f(x)}=|f(x)|^2.
\]
Moreover, $G$ has order at most $2$ because it is a product of two entire
functions of order at most $2$. Hence $|f|^2$ admits an entire extension of order at most $2$.  This completes the proof.

\end{proof}

In the following theorem, we show that a function $f\in V(h_n)$ generated by a Hermite function is conjugate phase retrieval, and it can be determined, up to a unimodular constant and conjugation, from the phaseless Hermite samples $|f(\gamma)|$ and $|f'(\gamma)|, \gamma\in\Gamma$, where  $\Gamma\subset\R$ is a sequentially compact sampling set.

\begin{thm}\label{hermite.cpr.prop}
{\rm Let $h_n$ be the Hermite function \eqref{hermite.def} of order $n$ and $V(h_n)$  be the shift-invariant space \eqref{h.sis} generated by $h_n$. Any function $f\in V(h_n)$  
can be determined, up to a unimodular constant and conjugation, from the magnitudes $|f(\gamma)|$ and $|f'(\gamma)|$, $\gamma\in\Gamma$, taken on a sequentially compact infinite set $\Gamma\subset \R$.} 
\end{thm}

 \begin{proof}
Take an arbitrary $f\in V(h_n)$.
By Proposition~\ref{hermite.lem}, the functions $f$  and $|f|^2$ admits an extension to an
entire function of finite order on the complex plane.


Following the same argument as in Proposition \ref{hermite.lem}, differentiating
\eqref{f.rep} yields
\[
f'(z)=\zeta_n\sum_{l=0}^n (-1)^l(\sqrt2)^l m_l\,F^{(l+1)}(z),
\]
and hence $f'$ is entire of order at most $2$. Consequently, $|f'|^2$ admits an
entire extension of order at most $2$.

Therefore, both entire functions $|f|^2$ and $|f'|^2$ are uniquely determined by
their values on the sequentially compact set $\Gamma$.
Together with Lemma~\ref{pr.entire.lem}, this implies that $f$ can be uniquely
recovered, up to a unimodular constant and conjugation, from the phaseless
samples $|f(\gamma)|$ and $|f'(\gamma)|$, $\gamma\in\Gamma$.
\end{proof}

\begin{remark}  In contrast to the sampling-theoretic approach employed in Theorem \ref{main.thm}, conjugate phase retrieval in shift-invariant spaces generated by Hermite functions is established using methods from complex analysis. In this setting, the uniqueness set for conjugate phase retrieval is not discrete and is used to uniquely determine an analytic function. In particular, previous results on conjugate phaseless sampling rely on Lemma \ref{sampling.lem}, which guarantees the existence of a suitable discrete sampling set, whereas the present result avoids this requirement altogether by exploiting analytic continuation rather than discrete sampling.

In the Gaussian case, the sampling-based argument requires an additional decay assumption \eqref{signal.assump} on the coefficient sequence to ensure that the derivative of the signal remains in the  Gaussian shift-invariant space, thereby allowing the sampling theory to be applied to the derivative, cf. \cite{GRS2020}. In  the Hermite case no such decay condition is needed. Indeed, as
shown in Proposition~\ref{hermite.lem}, both the signal and its derivatives admit
analytic extensions to entire functions of order at most $2$, and the same holds
for the squared modulus of the derivatives. This property eliminates
the need to verify that the derivative belongs to the original shift-invariant space. Consequently, the conjugate phase retrieval result in
Theorem \ref{hermite.cpr.prop} is established via a fundamentally different proof strategy from that of Theorem~\ref{main.thm}.

\end{remark}

\medskip

\section{Gaussian shift-invariant functions  with finite coefficients}\label{spec.sec}

\medskip 

Given a nonzero signal
\[
f(x)=\sum_{k\in\mathbb Z} c_k\, e^{-\lambda(x-\beta k)^2}\in V^{\infty}_{\beta,\lambda},
\]
define
\begin{equation}\label{Kpm.def}
K_-(f)=\inf\{k\in\mathbb Z:\ c_k\neq 0\},
\qquad
K_+(f)=\sup\{k\in\mathbb Z:\ c_k\neq 0\},
\end{equation}
and write $K_\pm:=K_\pm(f)$ when unambiguous.   
A function $f\in V^\infty_{\beta,\lambda}$ is said to be of \emph{finite duration} if
\[
-\infty<K_-(f)\le K_+(f)<\infty,
\]
that is, only finitely many coefficients $\{c_k\}$ are nonzero.

\medskip
In this section, we study conjugate phase retrieval for signals in
$V^\infty_{\beta,\lambda}$ of finite duration.
We show that such a signal can be uniquely determined, up to a unimodular constant
and conjugation, from finitely many magnitude measurements of the signal and its
derivative.
More precisely, in Theorem \ref{finite.cpr.thm} we prove that a signal of finite
duration can be recovered from $2(K_+(f)-K_-(f))+1$ phaseless Hermite samples.

\medskip

Moreover, the proof of Theorem \ref{finite.cpr.thm} is constructive and naturally
leads to an explicit algorithm for conjugate phaseless sampling and reconstruction
of Gaussian shift-invariant signals with finite duration.

\begin{thm}\label{finite.cpr.thm}
{\rm Let
\[
f(x)=\sum_{k=K_-(f)}^{K_+(f)} c_k\, e^{-\lambda(x-\beta k)^2}
\in V^\infty_{\beta,\lambda}
\]
be a signal satisfying \eqref{signal.assump}, where
$-\infty< K_-(f)\le K_+(f)<\infty$.
Let $\Gamma\subset\mathbb R$ be a sampling set consisting of
$ 2\bigl(K_+(f)-K_-(f)\bigr)+1$  
distinct points. Then the signal $f$ can be uniquely determined, up to a unimodular
constant and conjugation, from the phaseless samples
\[
\bigl\{|f(\gamma)|,\ |f'(\gamma)| : \gamma\in\Gamma\bigr\}.
\]}
\end{thm}

\medskip


\begin{proof}
\noindent\textbf{Step 1:  Reconstruct the modulus function $|f|$.}
\smallskip

 Set
 $\tilde c_k=c_ke^{-\lambda\beta^2k^2}, k\in\Z$. 
 By a direct calculation, we have 
 \begin{eqnarray}\label{f.abs.eq.1}
 |f(x)|^2 \hskip-0.05in&\hskip-0.1in=&\hskip-0.1in e^{-2\lambda x^2}\sum_{k=K_-}^{K_+}\sum_{j=K_-}^{K_+}{\tilde c}_k\bar{\tilde c}_je^{2\lambda\beta(k+j)x}
 =:e^{-2\lambda x^2}\sum_{m=2K_-}^{2K_+}A_me^{2\lambda\beta mx}
 \end{eqnarray}
 with the real coefficients
\begin{equation}\label{A.def.1}
A_m=\left\{
\begin{array}{ll}
\displaystyle \sum_{j=K_-}^{m-K_-}\tilde c_{m-j}\overline{\tilde c_j},
& 2K_-\le m\le K_-+K_+,\\[2mm]
\displaystyle \sum_{j=m-K_+}^{K_+}\tilde c_{m-j}\overline{\tilde c_j},
& K_-+K_+ +1\le m\le 2K_+.
\end{array}
\right.
\end{equation}
Hence, for each $\gamma\in\Gamma$,
\[
e^{2\lambda \gamma^2}|f(\gamma)|^2=\sum_{m=2K_-}^{2K_+}A_m e^{2\lambda\beta m\gamma}.
\]
Since $\Gamma$ consists of $2(K_+-K_-)+1$ distinct points, the Vandermonde matrix
\[
V:=\Big(e^{2\lambda\beta m\gamma}\Big)_{2K_-\le m\le 2K_+, \gamma\in\Gamma}
\]
is invertible. Therefore the coefficients $A_m$ for $2K_-\le m\le 2K_+$, and hence the function
$|f|^2$, are  uniquely determined from the values $\{e^{2\lambda\gamma^2}|f(\gamma)|^2:\gamma\in\Gamma\}$.

\smallskip
\medskip
\noindent\textbf{Step 2: Reconstruct the modulus function 
 $|f'|$.}
\smallskip

Define
\[
d(x):=\sum_{k=K_-}^{K_+} k c_k e^{-\lambda(x-\beta k)^2}.
\]
By \eqref{signal.assump}, we have $d\in V^\infty_{\beta,\lambda}$. A direct calculation yields
\begin{equation}\label{d.def}
|d(x)|^2=e^{-2\lambda x^2}\sum_{m=2K_-}^{2K_+}B_m e^{2\lambda\beta m x},
\end{equation}
where the real coefficients $B_m$ are
\begin{equation}\label{B.def.1}
B_m=\left\{
\begin{array}{ll}
\displaystyle \sum_{j=K_-}^{m-K_-}(m-j)j\,\tilde c_{m-j}\overline{\tilde c_j},
& 2K_-\le m\le K_-+K_+,\\[2mm]
\displaystyle \sum_{j=m-K_+}^{K_+}(m-j)j\,\tilde c_{m-j}\overline{\tilde c_j},
& K_-+K_+ +1\le m\le 2K_+.
\end{array}
\right.
\end{equation}

Since
\[
f'(x)=\sum_{k=K_-}^{K_+} c_k\cdot (-2\lambda)(x-\beta k)e^{-\lambda(x-\beta k)^2}
= -2\lambda x f(x)+2\lambda\beta\, d(x),
\]
we have
\begin{eqnarray}\label{f'.abs.def.1}
\frac{|f'(x)|^2}{4\lambda^2}
&=&
x^2|f(x)|^2+\beta^2|d(x)|^2-\beta x\,e^{-2\lambda x^2}\sum_{m=2K_-}^{2K_+}mA_m e^{2\lambda\beta m x}.
\end{eqnarray}
Therefore the values $\{|d(\gamma)|^2:\gamma\in\Gamma\}$ can be computed from phaseless Hermite samples and the already determined $A_m$ by
\begin{equation}\label{d.samples}
\beta^2|d(\gamma)|^2
=
\frac{|f'(\gamma)|^2}{4\lambda^2}-\gamma^2|f(\gamma)|^2
+\beta \gamma e^{-2\lambda\gamma^2}\sum_{m=2K_-}^{2K_+}mA_m e^{2\lambda\beta m\gamma},
\gamma\in\Gamma.
\end{equation}

By \eqref{d.def} and  the invertibility of the matrix $\begin{pmatrix}e^{2\lambda\beta m\gamma}\end{pmatrix}_{2K_-\le m\le 2K_+, \gamma \in \Gamma}$, the coefficients $B_m, 2K_-\le m\le 2K_+$, and hence the modulus function $|d|$ can be determined from  $|d(\gamma)|, \gamma\in\Gamma$. 
This together with \eqref{f'.abs.def.1} determines the modulus function $|f'|$.

\smallskip
\medskip
\noindent\textbf{Step 3: Reconstruct the coefficients $c_k$, $K_-\le k\le K_+$.}
\smallskip

We claim that the coefficients $\tilde c_k$, $K_-\le k\le K_+$, can be obtained from $\{A_m\}_{m=2K_-}^{2K_+}$ and $ \{B_m\}_{m=2K_-}^{2K_+}$ by induction.

Since $A_{2K_-}=|\tilde c_{K_-}|^2\neq 0$. Without loss of generality, we assume
\begin{equation}\label{c_0.def.1}
\tilde c_{K_-}=\sqrt{A_{2K_-}}\in \R^+,
\end{equation}
otherwise replace $f$ by $\alpha f$ with some unimodular constant $\alpha\in\T$.

Moreover, take $m=2K_-+1$ in \eqref{A.def.1}, then 
\begin{equation}\label{S3:ReK-1}
\Re\tilde c_{K_-+1}=\frac{A_{2K_-+1}}{2\tilde c_{K_-}}.
\end{equation}

\medskip
\noindent\textbf{Case (a): all coefficients up to $k-1$ are real.}

Fix $k$ with $K_-<k\le \lfloor\frac{K_-+K_+}{2}\rfloor$ and assume
\[
\tilde c_{K_-+1},\ldots,\tilde c_{k-1}\in\R.
\]
Take $m=K_-+k$ in \eqref{A.def.1} yields
\begin{equation}\label{S3:Rek-a}
\Re\tilde c_k=
\frac{A_{K_-+k}-\sum_{j=K_-+1}^{k-1}\tilde c_{K_-+k-j}\tilde c_j}
     {2\tilde c_{K_-}}.
\end{equation}
Moreover,  $\Re\tilde c_{l}, \ell\in \{k,k+1,\ldots,2k-K_--1\}$ can be obtained sequentially by taking  $m=K_-+\ell$. 

Next take $m=2k\le K_-+K_+$ in \eqref{A.def.1} and \eqref{B.def.1},
\[
\left\{
\begin{aligned}
A_{2k}-C_{2k}&=2\tilde c_{K_-}\Re\tilde c_{2k-K_-}+|\tilde c_k|^2,\\
B_{2k}-\widetilde C_{2k}&=2K_-(2k-K_-)\tilde c_{K_-}\Re\tilde c_{2k-K_-}+k^2|\tilde c_k|^2,
\end{aligned}
\right.
\]
where 
$C_{2k}:=2\sum_{j=K_-+1}^{k-1}\Re(\tilde c_{2k-j})\,\tilde c_j$ and 
$\widetilde C_{2k}:=2\sum_{j=K_-+1}^{k-1}j(2k-j)\Re(\tilde c_{2k-j})\,\tilde c_j.$

Then 
\begin{equation}\label{S3:absck}
|\tilde c_k|^2=
\frac{K_-(2k-K_-)(A_{2k}-C_{2k})-(B_{2k}-\widetilde C_{2k})}
     {K_-(2k-K_-)-k^2}.
\end{equation}
Consequently,
\[
|\Im\tilde c_k|^2=|\tilde c_k|^2-(\Re\tilde c_k)^2.
\]
If $|\Im\tilde c_k|^2=0$, then $\tilde c_k\in\R$ is uniquely determined; otherwise
$\tilde c_k$ is determined up to conjugation.

\medskip
\noindent\textbf{Case (b): the first nonreal coefficient occurs at index
\[
j_0:=\min\{j\in\{K_-+1,\ldots, K_+\}:\ \Im\tilde c_j\neq0\}.
\]}
Then $\tilde c_{K_-+1},\ldots,\tilde c_{j_0-1}\in\R$, and $\tilde c_{j_0}$ is already
determined up to conjugation from Case (a) applied at $k=j_0$. 

Now fix $k$ with $j_0<k\le \lfloor\frac{K_-+K_+}{2}\rfloor$ and assume that
$\tilde c_{K_-},\ldots,\tilde c_{k-1}$ have been determined.
We first reconstruct the real parts
\[
\Re\tilde c_\ell,\ \ell=k,k+1,\ldots,k+j_0-1-K_-,
\]
sequentially from $A_{K_-+\ell}$.

Take $m=K_-+\ell$ for $\ell\in \{k,k+1,\ldots,k+j_0-1-K_-\}$, then 
\begin{equation*}\label{step3:Re-seq}
A_{m}=2\sum_{j=K_-}^{j_0-1}\tilde c_{j}\Re\big(\tilde c_{m-j}\big)+\sum_{j=j_0}^{m-j_0}\tilde c_{m-j}\,\overline{\tilde c}_j,
\end{equation*}
then  $\Re\tilde c_{k},\Re\tilde c_{k+1},\ldots, \Re \tilde c_{k+j_0-1-K_-}$ are obtained
sequentially once the earlier coefficients are known.

\smallskip
Next, we determine the imaginary part of $\tilde c_k$. 

Take $m=k+j_0$ in \eqref{A.def.1} and \eqref{B.def.1}. Then
\[
\left\{
\begin{aligned}
A_m-R_m
&=2\tilde c_{K_-}\Re\tilde c_{m-K_-}
 +2\Re(\tilde c_{j_0}\overline{\tilde c}_k),\\
B_m-\widetilde R_m
&=2K_-(m-K_-)\tilde c_{K_-}\Re\tilde c_{m-K_-}
 +2j_0k\,\Re(\tilde c_k \overline{\tilde c}_{j_0}),
\end{aligned}
\right.
\]
where
\[
R_m:=2\sum_{j=K_-+1}^{j_0-1}{\tilde c}_{j}\Re\big( \tilde c_{m-j}\big)+\sum_{j=j_0+1}^{k-1}{\tilde c}_{m-j}\overline{\tilde c}_{j}\]
and 
\[\widetilde R_m
:=
2\sum_{j=K_-+1}^{j_0-1}j(m-j){\tilde c}_{j}\Re\big( \tilde c_{m-j}\big)+\sum_{j=j_0+1}^{k-1}j(m-j){\tilde c}_{m-j}\overline{\tilde c}_{j}.
\]
Eliminating $\Re\tilde c_{m-K_-}$ from the above system yields
\begin{equation}\label{S3:Reprod}
\Re(\tilde c_{j_0}\overline{\tilde c_k})
=
\frac{K_-(m-K_-)(A_m-R_m)-(B_m-\widetilde R_m)}
     {2\,[K_-(m-K_-)-j_0k]}.
\end{equation}

Consequently,
\begin{equation}\label{S3:Imck}
\Im\tilde c_k
=
\frac{\Re(\tilde c_{j_0}\overline{\tilde c_k})
-\Re(\tilde c_{j_0})\,\Re(\tilde c_k)}
{\Im(\tilde c_{j_0})}.
\end{equation}

\medskip
Proceeding for $k=K_-+1,\ldots,\lfloor\frac{K_-+K_+}{2}\rfloor$, we reconstruct the
first half of $\{\tilde c_k\}$. For $k\ge \frac{K_-+K_+}{2}$, the conclusion follows
by similar arguments.

\smallskip 

 Together with \eqref{c_0.def.1},
we have that $(\tilde c_{K_-},\ldots, \tilde c_{K_+} )$ is uniquely determined, up to a unimodular constant and conjugation, from $A_m, B_m, K_-\le m\le K_+$, and hence proves our conclusion. 
\end{proof}

\medskip

 If the signal $f\in \Vbeta$ has a real-valued coefficient sequence, that is,
\[
c_k \in \mathbb{R}, \  K_- \le k \le K_+ ,
\]
then the conjugate ambiguity present in the complex-valued setting does not arise. Indeed, a careful inspection of the proof of Theorem~\ref{finite.cpr.thm} shows that,
in the real-valued case, the quadratic quantities $A_k$, $2K_- \le k \le 2K_+$,
already determine the coefficients $\{c_k\}_{k=K_-}^{K_+}$ uniquely, up to a global sign. 
This leads to the following corollary on phase retrieval for real-valued functions in the Gaussian shift-invariant space.

\begin{corollary}[Real Case]\label{cor:real-coef-pr}
Let
\[
f(x)=\sum_{k=K_-(f)}^{K_+(f)} c_k\, e^{-\lambda(x-\beta k)^2}
\in V^\infty_{\beta,\lambda}
\]
be a signal with real coefficients $c_k\in\R$ for $K_-(f)\le k\le K_+(f)$, where
$-\infty< K_-(f)\le K_+(f)<\infty$. 
Let $\Gamma\subset\mathbb R$ be a sampling set consisting of
$2\bigl(K_+(f)-K_-(f)\bigr)+1$ distinct points. 
Then the signal $f$ can be uniquely determined, up to a global sign, from the phaseless samples
\[
\bigl\{|f(\gamma)|:\ \gamma\in\Gamma\bigr\}.
\]
\end{corollary}

\medskip 

Based on the constructive proof of Theorem \ref{finite.cpr.thm}, we propose the following reconstruction algorithm for a finite-duration signal $f(x)=\sum_{k=K_-}^{K_+} c_k e^{-(x-k)^2}\in V^{\infty}_{1,1}$.  Given the phaseless samples $
\big\{|f(\gamma_l)|,\ |f'(\gamma_l)|\big\}_{l=0}^{2(K_+-K_-)}$, 
the coefficient sequence
$\{c_k\}_{k=K_-}^{K_+}$ (and hence $f$) can be recovered as follows.


\begin{algorithm}[htp!]
\caption{Reconstruction of $\{ c_k\}_{k=K_-}^{K_+}$ from $\{A_m,B_m\}_{m=2K_-}^{2K_+}$}
\label{alg:step3-paper}
\begin{algorithmic}
\State {\bf Inputs:} Integers $K_-\le K_+$; real coefficients $\{A_m,B_m\}_{m=2K_-}^{2K_+}$.

\State \textbf{Initialization.} Set
\[
\tilde c_{K_-}=\sqrt{A_{2K_-}}\in\R^+,\qquad
\Re\tilde c_{K_-+1}=\frac{A_{2K_-+1}}{2\tilde c_{K_-}},\qquad
j_0=0.
\]

\State \textbf{Phase I (detect $j_0$ and fix the branch).}
\For{$k=K_-+1$ to $\lfloor (K_-+K_+)/2\rfloor$}

\State For $\ell=k,\ldots,2k-K_--1$, compute sequentially
\[
\Re\tilde c_\ell
=
\frac{A_{K_-+\ell}-\sum_{j=K_-+1}^{\ell-1}\tilde c_{K_-+\ell-j}\overline{\tilde c_j}}
     {2\tilde c_{K_-}}.
\]

\State Evaluate 
\[
C_{2k}=2\sum_{j=K_-+1}^{k-1}\Re(\tilde c_{2k-j})\,\tilde c_j,\qquad
\widetilde C_{2k}=2\sum_{j=K_-+1}^{k-1}j(2k-j)\Re(\tilde c_{2k-j})\,\tilde c_j,
\]
and compute
\[
|\tilde c_k|^2=
\frac{K_-(2k-K_-)(A_{2k}-C_{2k})-(B_{2k}-\widetilde C_{2k})}
     {K_-(2k-K_-)-k^2},
\qquad
|\Im\tilde c_k|^2=|\tilde c_k|^2-(\Re\tilde c_k)^2.
\]

\If{$|\Im\tilde c_k|^2>0$}
\State Set $j_0=k$ and choose the branch
\[
\Im\tilde c_{j_0}=+\sqrt{|\Im\tilde c_{j_0}|^2},\qquad
\tilde c_{j_0}=\Re\tilde c_{j_0}+i\,\Im\tilde c_{j_0}.
\]
\State \textbf{Stop}  \Comment{{\bf Output}: $\tilde c_{K_-},\ldots,\tilde c_{j_0-1}\in\R$, and $\Re\tilde c_\ell$ is known for $\ell\le 2j_0-K_--1$.}
\EndIf

\If{$|\Im\tilde c_k|^2=0$}
\State Set $\tilde c_k=\Re\tilde c_k\in\R$.
\EndIf

\EndFor

\State \textbf{Phase II (propagate imaginary parts if $j_0\neq 0$).}
\If{$j_0\neq 0$}
\For{$k=j_0+1$ to $\lfloor (K_-+K_+)/2\rfloor$}

\State \emph{(II-a)} For $\ell=k,\ldots,k+j_0-1-K_-$, compute $\Re\tilde c_\ell$ sequentially from $A_{K_-+\ell}$ using the same formula as in Phase~I.

\State \emph{(II-b)} Set $m=k+j_0$ and calcuate 
\[
R_m=
2\sum_{j=K_-+1}^{j_0-1}\tilde c_j\Re(\tilde c_{m-j})
+\sum_{j=j_0+1}^{k-1}\tilde c_{m-j}\overline{\tilde c_j},
\]
\[
\widetilde R_m=
2\sum_{j=K_-+1}^{j_0-1}j(m-j)\tilde c_j\Re(\tilde c_{m-j})
+\sum_{j=j_0+1}^{k-1}j(m-j)\tilde c_{m-j}\overline{\tilde c_j}.
\]
Compute
\[
\Re(\tilde c_{j_0}\overline{\tilde c_k})
=
\frac{K_-(m-K_-)(A_m-R_m)-(B_m-\widetilde R_m)}
     {2\,[K_-(m-K_-)-j_0k]} \ \ {\rm and} \  \
\Im\tilde c_k
=
\frac{\Re(\tilde c_{j_0}\overline{\tilde c_k})
-\Re\tilde c_{j_0}\Re\tilde c_k}{\Im\tilde c_{j_0}}. 
\]

\EndFor
\EndIf

\State {\bf Outputs:}  $ c_k=\tilde c_ke^{k^2}, K_-\le k\le K_+$.
\end{algorithmic}
\end{algorithm}

\bigskip 

{\bf Acknowledgments}:  The authors would like to thank the referees for their careful reading and helpful
comments, which improved the manuscript.

\appendix

\section{Proof of Lemma \ref{f.abs.lem}}
\label{lem.pr.sis}


Write
\[
f(x)=\sum_{k\in\mathbb{Z}} c_k e^{-\lambda(x-\beta k)^2},
\qquad c\in \ell^\infty(\mathbb{Z}).
\]
Then 
\begin{eqnarray*}
|f(x)|^2
&=&\sum_{k,l\in\mathbb{Z}} c_k\overline{c_l}
e^{-\lambda\beta^2(k^2+l^2)} e^{\lambda \beta^2 (k+l)^2/2}
e^{-2\lambda\left(x-\frac{\beta(k+l)}{2}\right)^2} \nonumber \\
&=&\sum_{n\in\mathbb{Z}}
\left(\sum_{k\in\mathbb{Z}} c_k \overline{c_{n-k}}
e^{-\lambda\beta^2 k^2} e^{-\lambda\beta^2(n-k)^2}\right)
e^{\lambda\beta^2 n^2/2} e^{-2\lambda(x-(n\beta)/2)^2}.
\end{eqnarray*}

Set
\begin{equation*}
d_k=c_k e^{-\lambda\beta^2 k^2}, 
r_n=\sum_{k\in\mathbb{Z}} c_k \overline{c_{n-k}}
e^{-\lambda\beta^2 k^2} e^{-\lambda\beta^2(n-k)^2}, \ {\rm and } \ 
\tilde r_n = r_n e^{\lambda\beta^2 n^2/2}.
\end{equation*}
From these definitions we see that
\begin{equation*}
r=d*\overline{d}
\quad\text{and}\quad
|f(x)|^2=\sum_{n\in\mathbb{Z}} \tilde r_n
e^{-2\lambda(x-(n\beta)/2)^2}.
\end{equation*}

If $c\in \ell^\infty(\mathbb{Z})$, then $\tilde r\in \ell^\infty(\mathbb{Z})$,
because
\begin{eqnarray*}
|\tilde r_n|
&=&\left|
\sum_{k\in\mathbb{Z}} c_k\overline{c_{n-k}}
e^{-\lambda\beta^2  k^2} e^{-\lambda\beta^2(n-k)^2}
\right| e^{\lambda\beta^2 n^2/2} \\
&\le& \|c\|_\infty^2
\sum_{k\in\mathbb{Z}} e^{-\lambda\beta^2(k^2+(n-k)^2-n^2/2)} =\|c\|_\infty^2
\sum_{k\in\mathbb{Z}} e^{-\lambda\beta^2 (n-2k)^2/2}.
\end{eqnarray*}
Setting
\[
C=\max\!\left(
\sum_{k\in\mathbb{Z}} e^{-2\lambda\beta^2 k^2},
\;
\sum_{k\in\mathbb{Z}} e^{-\lambda\beta^2(1-2k)^2/2}
\right),
\]
we have shown that
\begin{equation*}
\|\tilde r\|_\infty \le C \|c\|_\infty^2.
\end{equation*}
As a consequence $|f|^2\in V_{\frac{\beta}{2}, 2\lambda}^\infty$, as claimed.

\section{Proof of Lemma \ref{pr.entire.lem}}\label{appendix.sec}
 To prove Lemma \ref{pr.entire.lem}, we recall the well-known  Hadamard factorization theorem.

\begin{lem}\label{Hadamard.lem}{\rm(}Hadamard Factorization Theorem{\rm)} {\rm\cite{Branges} Let $f$ be an entire function of finite order $\rho$. Then $f$ has the unique factorization $$f(z)=z^{m}e^{P(z)}\prod_{a\in Z_f}E_{[\rho]}(a,z),$$
where $m\in \N$ is the multipliity of the zero at $z=0$, $P(z)$ is a polynomial of degree at most $\rho$, $Z_f$ is the multiset of nonzero zeros of $f$, counted with the multiplicity,  $[\rho]$ is the largest integer number not greater than $\rho$ and 
\begin{equation*}
 E_{[\rho]}(a,z)=
\begin{cases}
(1-\frac{z}{a} )& {\rm \ if}\quad [\rho]=0,\\
(1-\frac{z}{a})e^{(\frac{z}{a}+\frac{1}{2}\frac{z^{2}}{a^{2}}+\cdot\cdot\cdot+\frac{1}{[\rho]}(\frac{z}{a})^{[\rho]})}& {\rm otherwise}.
\end{cases}
\end{equation*} }
\end{lem}

\medskip 
\begin{proof}[Proof of Lemma \ref{pr.entire.lem}]
Let $f, g$ be the entire functions of finite order satisfying \eqref{f.lemma.eq} and denote   $Z_f$ and $Z_g$ their zero multisets with $0$ excluded respectively.  We first prove that there exist entire functions $u, v$ of finite order  such that $f=uv$ and $g=uv^\natural$. As $ff^\natural=gg^\natural$, the entire functions $f$ and $g$ are of the same order $\rho$.
  By Lemma \ref{Hadamard.lem},  the functions $f$ and $g$ can be decomposed as
 \begin{equation*}
 f(z)=z^{m}e^{P_f(z)}\prod_{a\in Z_f}E_{[\rho]}(a,z),
 \end{equation*}
 and
 \begin{equation*}
 g(z)=z^{n}e^{P_g(z)}\prod_{a\in Z_g}E_{[\rho]}(a,z).
 \end{equation*} Then their involution functions
  \begin{equation*}
 f^\natural(z)=(-1)^mz^{m}e^{P_f^\natural (z)}\prod_{a\in Z_f}E_{[\rho]}(-\bar a,z),
 \end{equation*}
 and
 \begin{equation*}
 g^\natural(z)=(-1)^nz^{n}e^{P_g^\natural(z)}\prod_{a\in Z_g}E_{[\rho]}(-\bar a,z).
 \end{equation*}
 By $ff^\natural =gg^\natural$, we have
 \begin{equation}\label{coeff.1}
 m=n, P_f+P_f^\natural=P_g+P_g^\natural{\rm \  and\ }Z_f\dot{\cup} Z_f^\natural =Z_g \dot{\cup} Z_g^\natural.
 \end{equation}

  Let $Z_0=Z_f\dot{\cap }Z_g$  be the multiset of the common elements in $Z_f$ and $Z_g$ with the multiplicities counted and $Z_1=Z_f\dot{\setminus }Z_0$ and $Z_2=Z_g\dot{\setminus} Z_0$.  By \eqref{coeff.1}, we have $Z_2=Z_1^\natural$. Define
 \begin{equation*}
 u(z)=z^m e^{\frac{P_f(z)+P_g(z)}{2}}\prod_{a\in Z_0} E_{[\rho]}(a,z){\rm\ and\ } v(z)=e^{\frac{P_f(z)-P_g(z)}{2}}\prod_{a\in Z_1} E_{[\rho]}(a,z).
 \end{equation*}
 Then we have \begin{equation}\label{f.decomp}
 f=uv{\rm\  and\ }g=uv^\natural.
 \end{equation}
 By $f'(f')^\natural=g'(g')^\natural$ and $(f^\natural)'=-(f')^\natural$,  a straightforward calculation based on  \eqref{f.decomp} yields 
\begin{equation*}
(u'u^\natural-u(u^\natural)')(v'v^\natural-v(v^\natural)')=0.
\end{equation*}
It implies that $u=\alpha u^\natural$ or $v=\alpha v^\natural$ for some unimodular  constant $\alpha\in\T$. This completes our proof.
 \end{proof}

\end{document}